\documentclass[11pt]{amsart}

\usepackage[english]{babel}
\usepackage{amsmath}
\usepackage{amsfonts}
\usepackage{amssymb}
\usepackage{amsthm}
\usepackage{epsfig}
\usepackage{graphicx}
\usepackage{url,hyperref}

\usepackage[strict]{changepage}

\newtheorem{teor}{Theorem}[section]

\newtheorem{corollary}[teor]{Corollary}
\newtheorem{prop}[teor]{Proposition}

\newtheorem*{remark*}{Remark}

\newcounter{theor}

\newtheorem{theorem}[theor]{Theorem}

\theoremstyle{definition}
\newtheorem{definition}[teor]{Definition}

\def\K{\mathcal{K}}

\def\L{\mathcal{L}}

\def\R{\mathbb{R}}

\def\vol{\mathrm{vol}}

\def\Stsym#1{S_{\!{#1}}}

\newcommand{\enorm}[1]{\left|{#1}\right|}
\newcommand{\dlat}{\mathrm{d}}

\def\proj{\mathrm{proj}}
\def\epi{\mathrm{epi}}
\def\stepi{\mathrm{epi}_s}
\def\sthyp{\mathrm{hyp}_s}

\numberwithin{equation}{section}

\begin{document}

\title{On a linear refinement of the Pr\'ekopa-Leindler inequality}

\author{A. Colesanti}
\address{Dipartimento di Matematica ``U. Dini'', Viale Morgagni 67/A,
50134-Firenze, Italy}
\email{colesant@math.unifi.it}
\author{E. Saor\'in G\'omez}
\address{Institut f\"ur Algebra und Geometrie, Otto-von-Guericke Universit\"at Magdeburg,
Universit\"atsplatz 2, D-39106-Magdeburg, Germany}
\email{eugenia.saorin@ovgu.de}
\author{J. Yepes Nicol\'as}
\address{Departamento de Matem\'aticas, Universidad de Murcia, Campus de
Espinar\-do, 30100-Murcia, Spain}
\email{jesus.yepes@um.es}

\thanks{Second author is supported by Direcci\'on General de Investigaci\'on MTM2011-25377
MCIT and FEDER.
Third author is supported by MINECO project MTM2012-34037.}



\begin{abstract}
If $f,g:\R^n\longrightarrow\R_{\geq0}$ are non-negative measurable functions, then the Pr\'ekopa-Leindler inequality asserts that the integral of the Asplund sum (provided that it is measurable) is greater or equal than the $0$-mean of the integrals of $f$ and $g$.
In this paper we prove that under the sole assumption that $f$ and $g$ have
a common projection onto a hyperplane, the Pr\'ekopa-Leindler inequality admits a linear refinement. Moreover, the same inequality can be obtained when assuming that both projections (not necessarily equal as functions) have the same integral. An analogous approach may be also carried out for the so-called Borell-Brascamp-Lieb inequality.
\end{abstract}

\maketitle

\section{Introduction}

The Pr\'ekopa-Leindler inequality, originally proved in
\cite{Prekopa} and \cite{Leindler}, states that if $\lambda\in (0,1)$  and $f,g,h:\R^n\longrightarrow\R_{\geq0}$ are non-negative measurable
functions such that, for any $x,y\in\R^n$,
\begin{equation}\label{e:PrekopaLeindlerCondicion}
h\bigl((1-\lambda) x+\lambda y\bigr)\geq f(x)^{1-\lambda}g(y)^{\lambda},
\end{equation}
then
\begin{equation*}
\int_{\R^n}h\,\dlat x\geq
\left(\int_{\R^n}f\,\dlat x\right)^{1-\lambda}\left(\int_{\R^n}g\,\dlat x\right)^{\lambda}.
\end{equation*}
A more stringent version of this result can be obtained considering the smallest function $h$ verifying condition \eqref{e:PrekopaLeindlerCondicion},
given $f$, $g$ and $\lambda$. Such a function is nothing but the so-called {\em Asplund sum} of $f$ and $g$, defined as follows (see e.g. \cite[p.~517]{Sch2}).

\begin{definition}
Given two non-negative functions $f,g:\R^n\longrightarrow\R_{\geq0}$ and $\lambda\in(0,1)$,
the function $(1-\lambda) f\star\lambda g:\R^n\longrightarrow\R_{\geq0}\cup\{\infty\}$ is defined by
\begin{equation*}
(1-\lambda) f\star\lambda g \,(x)=\sup_{(1-\lambda) x_1+\lambda x_2=x} f(x_1)^{1-\lambda} g(x_2)^{\lambda}.
\end{equation*}
\end{definition}

It is worth noting that the assumption that $f$ and $g$ are measurable
is not sufficient to guarantee that $(1-\lambda) f\star\lambda g$ is measurable; see
\cite[Section 10]{G}. The Pr\'ekopa-Leindler inequality can be now rephrased in the following way.

\begin{theorem}[Pr\'ekopa-Leindler Inequality]\label{t:PrekopaLeindler}
Let $\lambda\in (0,1)$ and let $f,g:\R^n\longrightarrow\R_{\geq0}$ be non-negative measurable
functions such that $(1-\lambda) f\star\lambda g$ is measurable as well. Then
\begin{equation}\label{e:PrekopaLeindler}
\int_{\R^n}(1-\lambda) f\star\lambda g\,\dlat x\geq
\left(\int_{\R^n}f\,\dlat x\right)^{1-\lambda}\left(\int_{\R^n}g\,\dlat x\right)^{\lambda}.
\end{equation}
\end{theorem}

This functional inequality can be seen as the analytic counterpart of a geometric inequality, i.e., the Brunn-Minkowski inequality.
Let $\lambda\in[0,1]$ and let $A$ and $B$ be two nonempty (Lebesgue) measurable subsets of $\R^n$, such that their vector linear combination
$$
(1-\lambda)A+\lambda B=\{(1-\lambda)x+\lambda y\,:\,x\in A\,,\,y\in B\}
$$
is also measurable. Then
\begin{equation}\label{BM1}
\vol((1-\lambda)A+\lambda B)^{1/n}\ge
(1-\lambda)\vol(A)^{1/n}+\lambda\vol(B)^{1/n}
\end{equation}
where by $\vol(\cdot)$ we denote the Lebesgue measure. The Brunn-Minkowski inequality admits an equivalent form,
often referred to as its multiplicative or dimension free version:
\begin{equation}\label{e:mult_B-M_ineq}
\vol\bigl((1-\lambda)K+\lambda
L\bigr)\geq\vol(K)^{1-\lambda}\vol(L)^{\lambda}.
\end{equation}
Note that a straightforward proof of \eqref{e:mult_B-M_ineq} can be obtained by applying
\eqref{e:PrekopaLeindler} to characteristic functions. Indeed, for $A,B\subset\R^n$,
\begin{equation}\label{e:AsplCaracFunc}
(1-\lambda) \chi_{_A}\star\lambda \chi_{_B}=\chi_{_{(1-\lambda) A+\lambda B}}
\end{equation}
where $\chi$ denotes the characteristic function. On the other hand, the Pr\'ekopa-Leindler
inequality can be proved by induction on the dimension $n$, and the initial case $n=1$ follows
easily from \eqref{BM1} (see for instance \cite[p.~3]{Pisier} or \cite[Section 7]{G}).

Inequalities \eqref{e:PrekopaLeindler} and \eqref{BM1} have a strong link with convexity, as is shown by
the description of the equality conditions. Indeed, equality may occur in \eqref{e:PrekopaLeindler} if and only if,
roughly speaking, there exists a {\it log-concave} function $F$ (i.e., $F=e^{-u}$,
where $u$ is convex), such that $f$, $g$ and $h$ coincide a.e. with $F$, up to translations and rescaling of the coordinates
(see \cite{Dub}). As a consequence, equality holds in the Brunn-Minkowski inequality if and only if $A$ and $B$ are two
homothetic compact convex sets, up to subsets negligible with respect to the Lebesgue measure.

The Brunn-Minkowski inequality is one of the most powerful results in convex geometry and, together with its analytic
companion \eqref{e:PrekopaLeindler} has a wide range of applications in analysis, probability, information theory and other areas
of mathematics. We refer the reader to the updated monograph \cite{Sch2}, entirely devoted to convex geometry, and to the extensive
and detailed survey \cite{G} concerning the Brunn-Minkowski inequality.

\medskip

In \cite[Section 50]{BF}, linear refinements of the Brunn-Minkowski inequality
were obtained for {\it convex bodies} (i.e., nonempty compact convex sets) having a common
orthogonal projection onto a hyperplane, or more generally a
projection with the same $(n-1)$-dimensional volume. To state these results we need some notation:
$\K^n$ denotes the set of convex bodies in $\R^n$; $\L^n_{n-1}$ is the set of
$(n-1)$-dimensional subspaces of $\R^n$ (i.e., hyperplanes containing the origin) and, given $K\in\K^n$ and $H\in\L^n_{n-1}$,
$K|H$ is the orthogonal projection of $K$ onto $H$ (which is a convex body as well). Moreover,
$\vol_{n-1}(\cdot)$ denotes the $(n-1)$-dimensional Lebesgue measure in $\R^n$.

\begin{theorem}[{\cite{BF,Gi}}]\label{t:BMproye}
Let $K,L\in\K^n$ be convex bodies such that there exists
$H\in\L^n_{n-1}$ with $K|H=L|H$. Then, for all $\lambda\in[0,1]$,
\begin{equation*}\label{e:BMproye}
\vol\bigl((1-\lambda)K+\lambda L\bigr)
\geq(1-\lambda)\vol(K)+\lambda\vol(L).
\end{equation*}
\end{theorem}

In other words, the volume itself is a concave function on the ``segment'' joining $K$ and $L$ in $\K^n$.

\begin{theorem}[{\cite{BF,Gi,Oh}}]\label{t:BMproye_vol}
Let $K,L\in\K^n$ be convex bodies such that there exists
$H\in\L^n_{n-1}$ with $\vol_{n-1}(K|H)=\vol_{n-1}(L|H)$. Then, for all
$\lambda\in[0,1]$,
\begin{equation}\label{e:BMproye_vol}
\vol\bigl((1-\lambda)K+\lambda L\bigr)
\geq(1-\lambda)\vol(K)+\lambda\vol(L).
\end{equation}
\end{theorem}

These results have been extended to compact sets  in \cite{Oh} and more recently
in \cite[Subsection 1.2.4]{Gi}; see also \cite{HCYN2} for related topics.

We would like to point out that, contrary to Theorem \ref{t:BMproye}, Theorem \ref{t:BMproye_vol} does not provide the concavity of the function $f(\lambda)=\vol\bigl((1-\lambda)K+\lambda L\bigr)$ for $\lambda\in[0,1]$. More precisely, inequality \eqref{e:BMproye_vol} only yields the condition $f(\lambda)\geq(1-\lambda)f(0)+\lambda f(1)$; nevertheless, when working with convex bodies $K$ and $L$ having a common projection onto a hyperplane, it is easy to check that the above condition implies indeed concavity of $f$ (see e.g. the proof of \cite[Theorem~2.1.3]{HCYN2}).
On the other hand, in the paper \cite{Di}, Diskant constructed an example
where the above-mentioned function is not concave under the sole assumption of a common volume projection (the bodies used by Diskant are essentially a {\em cap body} of a ball and a half-ball). For further details about this topic we refer to Notes for Section $7.7$ in \cite{Sch2} and the references therein.

\medskip

At this point it is natural to wonder whether analogous results to
Theorems \ref{t:BMproye} and \ref{t:BMproye_vol} could be obtained for Pr\'ekopa-Leindler inequality. The aim of this paper
is to provide an answer to this question. As a first step we notice that there is a rather natural way to define the ``projection''
of a function (see for instance \cite{KM}).

\begin{definition}
Given $f:\R^n\longrightarrow\R_{\geq0}\cup\{\infty\}$ and $H\in\L^n_{n-1}$, the \emph{projection} of $f$ onto $H$ is the (extended) function
$\proj_H(f)\,:\,H\longrightarrow\R_{\geq0}\cup\{\infty\}$ defined by
\begin{equation*}
\proj_H(f)(h)=\sup_{\alpha\in\R} f(h+\alpha \nu)
\end{equation*}
for $h\in H$, where $\nu$ is a normal unit vector of $H$.
\end{definition}

The geometric idea behind this definition is very simple: the hypograph of the projection
of $f$ onto $H$ is the projection of the hypograph of $f$ onto $H$. In particular, the projection of the characteristic function of a set $A$ is
just the characteristic function of the projection of $A$.

In this paper we show that under the equal projection assumption for the functions $f$ and $g$, the Pr\'ekopa-Leindler inequality becomes linear in $\lambda$. This is the analytical counterpart of Theorem \ref{t:BMproye}. Indeed, taking $f=\chi_{_K}$ and $g=\chi_{_L}$,
Theorem \ref{t:BMproye} may be obtained as a corollary.

\begin{teor}\label{t:PL-commonproj} Let $\lambda\in(0,1)$ and let $f,g:\R^n\longrightarrow\R_{\geq0}$ be non-negative measurable functions
such that $(1-\lambda) f\star\lambda g$ is measurable.
If there exists
$H\in\L^n_{n-1}$ such that
\begin{equation*}
  \proj_H(f)=\proj_H(g)
\end{equation*}
then
\begin{equation}\label{e:PL-commonproj}
   \int_{\R^n}(1-\lambda) f\star\lambda g\,\dlat x \geq (1-\lambda)\int_{\R^n}f\,\dlat x +\lambda \int_{\R^n}g\,\dlat x.
\end{equation}
\end{teor}

Notice that, by means of the Arithmetic-Geometric mean inequality, the Pr\'e\-kopa-Leindler inequality (Theorem \ref{t:PrekopaLeindler}) directly follows from the above result
(and hence, indeed, \eqref{e:PL-commonproj} is a stronger inequality under the common projection assumption). As we already remarked, the Pr\'ekopa-Leindler inequality
is naturally connected to log-concave functions, i.e., functions of the form $e^{-u}$ where $u\,:\,\R^n\longrightarrow\R\cup\{\infty\}$ is convex.
As a consequence of the above result we have the following statement.

\begin{corollary} Let $f,g:\R^n\longrightarrow\R_{\geq0}$ be log-concave functions and let $\lambda\in(0,1)$.
\noindent If there exists
$H\in\L^n_{n-1}$ such that
\begin{equation*}
  \proj_H(f)=\proj_H(g)
\end{equation*}
then
\begin{equation*}
   \int_{\R^n}(1-\lambda) f\star\lambda g\,\dlat x \geq (1-\lambda)\int_{\R^n}f\,\dlat x +\lambda\int_{\R^n}g\,\dlat x.
\end{equation*}
\end{corollary}

\begin{proof}
The Asplund sum
preserves log-concavity, as it easily follows with the so-called infimal convolution (see Section \ref{s:main}), and hence measurability (a log-concave function $\phi$ defined in $\R^n$ is, in fact, continuous in the interior of $\{x\in\R^n:\,\phi(x)>0\}$).
\end{proof}

We prove that \eqref{e:PL-commonproj} can be obtained under the less restrictive hypothesis that the integral of the projections coincide,
establishing a functional version of Theorem \ref{t:BMproye_vol}. In the general case of measurable $f$ and $g$ (see Theorem \ref{t:PLcommon_int_pro}
in Section 3), this result requires two mild (but technical) measurability assumptions. For simplicity, here we present this result for log-concave functions decaying to zero at infinity.

\begin{teor}\label{t:PLcommon_int_pro_log-concave}
Let $f,g:\R^n\longrightarrow\R_{\ge0}$ be log-concave functions
such that
$$
\lim_{|x|\to\infty}f(x)=\lim_{|x|\to\infty}g(x)=0,
$$
and let $\lambda\in(0,1)$. If there exists $H\in\L^n_{n-1}$ such that
\begin{equation*}\label{e:hyp_intprojcom}
 \int_H \proj_H(f)(x)\,\dlat x=\int_H \proj_H(g)(x)\,\dlat x<\infty
\end{equation*}
then
\begin{equation*}
   \int_{\R^n}(1-\lambda) f\star\lambda g\,\dlat x \geq(1-\lambda)\int_{\R^n}f\,\dlat x +\lambda \int_{\R^n}g\,\dlat x.
\end{equation*}
\end{teor}

\medskip

In the special case $n=1$, Theorem \ref{t:PL-commonproj} reduces to the following fact: if $f$ and $g$ are non-negative measurable functions
defined on $\R$ such that $(1-\lambda) f\star\lambda g$ is measurable and
$$
\sup_\R f=\sup_\R g,
$$
then \eqref{e:PL-commonproj} holds. This result can be found in \cite[Theorem 3.1]{BL}.

The Pr\'ekopa-Leindler inequality has been generalized by introducing $p$th me\-ans (see Section \ref{s:Extension to Borell-Brascamp-Lieb inequalities} for detailed definitions and explanations) on both sides of \eqref{e:PrekopaLeindler};
the resulting inequalities came to be called Borell-Brascamp-Lieb inequalities due to \cite{Borell} and \cite{BL}.
We have been able to extend our approach to these inequalities by obtaining the suitable versions of Theorems \ref{t:PL-commonproj} and \ref{t:PLcommon_int_pro_log-concave}. Again, for simplicity (and in order to avoid technical
measurability assumptions), we present here the result for the
case of $p$-concave functions (see Section \ref{s:Extension to Borell-Brascamp-Lieb inequalities} for the definition).

\begin{teor}\label{t:BBLcommon_int_pro_p-concave}
Let $f,g:\R^n\longrightarrow\R_{\geq0}\cup\{\infty\}$ be $p$-concave functions, where
$-1/n\leq p\leq\infty$, and let $\lambda\in(0,1)$. If there exists $H\in\L^n_{n-1}$ such that
\begin{equation*}
 \int_H \proj_H(f)(x)\,\dlat x=\int_H \proj_H(g)(x)\,\dlat x
\end{equation*}
then
\begin{equation*}
\int_{\R^n}(1-\lambda) f\star_p\lambda g\,\dlat x \geq(1-\lambda)\int_{\R^n}f\,\dlat x +\lambda\int_{\R^n}g\,\dlat x.
\end{equation*}
\end{teor}

\medskip

The paper is organized as follows. Section \ref{s:background} is devoted to collecting some definitions and preliminary constructions, whereas Theorems \ref{t:PL-commonproj} and \ref{t:PLcommon_int_pro_log-concave} (in fact, a more general version of the latter) will be proven in Section \ref{s:main}, as well as other related results. Finally in Section \ref{s:Extension to Borell-Brascamp-Lieb inequalities} we deal with the Borell-Brascamp-Lieb extensions, proving (among other results) Theorem \ref{t:BBLcommon_int_pro_p-concave}.

\section{Background material and auxiliary results}\label{s:background}

We work in the $n$-dimensional Euclidean space $\R^n$, $n\ge 1$, and $\K^n$ denotes the set of all convex bodies in $\R^n$. Given a subset $A$ of $\R^n$, $\chi_A$ is
the characteristic function of $A$.

With $\L^n_k$, $k\in\{0,1,2,\dots,n\}$, we will represent the set of
all $k$-dimensional linear subspaces of $\R^n$. For $H\in\L^n_k$, $H^{\bot}\in\L^n_{n-k}$ denotes
the orthogonal complement of $H$. Given $A\subset\R^n$ and $H\in\L^n_k$, the orthogonal
projection of $A$ onto $H$ is denoted by $A|H$.  For $k\in\{0,1,\dots,n\}$ and $A\subset\R^n$,
$\vol_k(A)$ denotes the $k$-dimensional Lebesgue measure of $A$ (assuming that $A$ is measurable
with respect to this measure). We will often omit the index $k$ when it is equal to the dimension $n$ of the
ambient space; in this case $\vol(\cdot)=\vol_n(\cdot)$ is just the ($n$-dimensional) Lebesgue measure.

Let $f:\R^n\longrightarrow\R$; we define the {\em strict epigraph} of $f$ by
\begin{equation*}
\stepi(f)=\{(x,t):\, x\in\R^n, \, t\in\R,\, t>f(x)\}\subset\R^{n+1},
\end{equation*}
while its {\it strict hypograph} (or subgraph) will be denoted as
\begin{equation*}
\sthyp(f)=\{(x,t):\, x\in\R^n, \, t\in\R,\, f(x)>t\}\subset\R^{n+1}.
\end{equation*}
The same definitions are valid for functions that take values on $\R\cup\{\infty\}$ (or $\R\cup\{\pm\infty\}$). In this case the strict epigraph of a function $f$ is empty if and only if $f$ is identically equal to infinity. Since we will work with non-negative functions, we also define
$$
\sthyp^+(f)=\{(x,t):\, x\in\R^n, \, t\in\R_{\geq0},\, f(x)>t\}\subset\R^{n+1}.
$$

Throughout this paper, given $H\in\L^n_{n-1}$, we set $\widetilde{H}=\nolinebreak H\times\nolinebreak\R$; i.e.,
$\widetilde{H}$ is the $n$-dimensional subspace (in $\R^{n+1}$) \emph{associated to} $H$ when working with epigraphs and hypographs of functions.

\medskip

The proof of our main result is based on symmetrization procedures; in fact we will use two distinct types of symmetrization of functions that will be introduced in the rest of this section.

\subsection{The Steiner symmetrization of a function}

To begin with, we briefly recall the Steiner symmetrization of sets:
given a nonempty measurable set $A\subset\R^n$ and $H\in\L^n_{n-1}$, the \emph{Steiner symmetral} of $A$ with respect to
$H$ is given by
\begin{equation*}
\Stsym{H}(A)=\left\{h+l\in\R^n: \, h\in A|H, \, l\in H^\perp, \, \enorm{l}\leq\frac{1}{2}\vol_1\bigl(A\cap (h+H^\perp)\bigr)\right\}
\end{equation*}
(see e.g. \cite[p.~169]{Gr}, or \cite[Section 9]{BF} for the compact convex case).
Notice that $\Stsym{H}(A)$ is well defined since the sections of a measurable set are also measurable
(for $\vol_{n-1}$ a.e. $h\in A|H$), and it is measurable (see for instance \cite[p. 67]{EG}). To complete the picture, we define $\Stsym{H}(\emptyset)=\emptyset$.

Next we define the first type of symmetrization of a (non-negative measurable) function $f$ with respect to a hyperplane $H$; roughly speaking
this is simply obtained by the Steiner symmetrization of the hypograph of $f$ with respect to $H$.  This technique is very well known in the theory of
partial differential equations and calculus of variations; see for instance \cite{Kawohl}.

Given a measurable $f:\R^n\longrightarrow\R_{\geq0}\cup\{\infty\}$, it will be convenient to write it in the form $f=e^{-u}$ where
$u:\R^n\longrightarrow\R\cup\{\pm\infty\}$ is a measurable function (i.e., $u(x)=-\log f(x)$ with the
conventions that $\log 0=-\infty$ and $\log \infty=\infty$). First we will consider the ``symmetral'' $u_H$ of $u$ with respect to $H\in\L^n_{n-1}$, which is given by
\begin{equation}\label{e:defu_H}
\stepi(u_H)=\Stsym{\widetilde{H}}\bigl(\stepi(u)\bigr).
\end{equation}

Notice that \eqref{e:defu_H} defines $u_H$ completely. Indeed, as $u$ is measurable its strict epigraph is also measurable and hence
$\Stsym{\widetilde{H}}\bigl(\stepi(u)\bigr)$ is well defined and measurable. Moreover it is easy to see that if a point $(x,\bar t)\in\Stsym{\widetilde{H}}\bigl(\stepi(u)\bigr)$,
with $x\in\R^n$ and $\bar t\in\R$, then the entire ``vertical'' half line $\{(x,t)\,:\,t\ge\bar t\}$ above it is contained in $\Stsym{\widetilde{H}}\bigl(\stepi(u)\bigr)$. Hence \eqref{e:defu_H} is equivalent to the following
explicit expression for $u_H:\R^n\longrightarrow\R\cup\{\pm\infty\}$:
\begin{equation*}\label{e:defu_H2}
u_H(x)=\inf\left\{r\in\R: \, (x,r)\in\Stsym{\widetilde{H}}\bigl(\stepi(u)\bigr)\right\}.
\end{equation*}
Note also that the measurability of its epigraph implies the measurability of $u_H$.
The next step is to define the symmetral of $f=e^{-u}$ through the symmetral of $u$, as follows.

\begin{definition}
Let $u:\R^n\longrightarrow\R\cup\{\pm\infty\}$ be a measurable function  and let $H\in\L^n_{n-1}$. Then the \emph{Steiner symmetral} of $f=e^{-u}$ is
$\Stsym{H}(f)= e^{-u_H}$, where $u_H$ is given by \eqref{e:defu_H}.
\end{definition}
As $u_H$ is measurable, $\Stsym{H}(f)$ is measurable as well.
Moreover, since $t\mapsto e^{-t}$ is a decreasing bijection between $\R\cup\{\pm\infty\}$ and $\R_{\geq0}\cup\{\infty\}$, we also have
\[\sthyp(\Stsym{H}(f))=\Stsym{\widetilde{H}}\bigl(\sthyp(f)\bigr)\] (which would have allowed us to define $\Stsym{H}(f)$ directly).
The above equality still holds if hypographs are replaced by positive hypographs or even if we consider sections
$(h+H^\perp)\times\R$, for any $h\in H$ (see Figure \ref{figurasimfunciones}), i.e., we replace $f$ by any of its restrictions
to a line perpendicular to $H$.

\medskip

\begin{figure}[h]
\begin{adjustwidth}{-20mm}{-20mm}
\hspace{0.7cm}
\includegraphics[width=5.25cm]{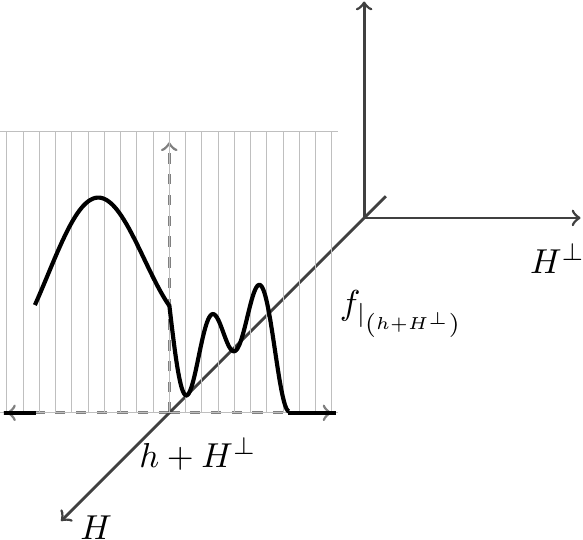}
\hspace{-1.7cm}
\includegraphics[width=5.75cm]{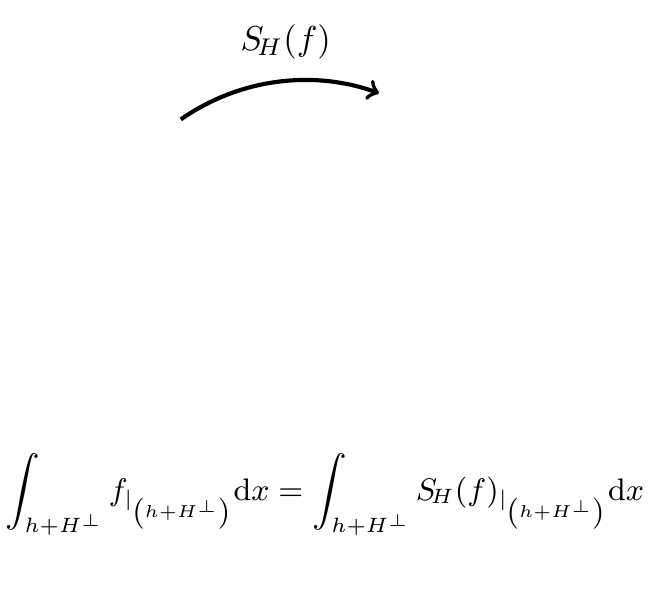}
\hspace{0cm}
\includegraphics[width=5.25cm]{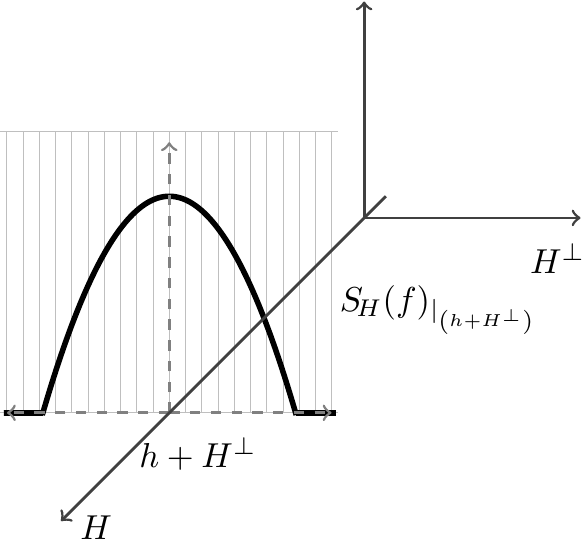}
\caption{}\label{figurasimfunciones}
\end{adjustwidth}
\end{figure}

\medskip

From the definition of $\Stsym{\widetilde{H}}(\cdot)$ we clearly have
\begin{equation*}
\stepi(u)|\widetilde{H}=\Stsym{\widetilde{H}}(\stepi(u))|\widetilde{H}=\Stsym{\widetilde{H}}(\stepi(u))\cap\widetilde{H},
\end{equation*}
which is equivalent to
\begin{equation*}
\sthyp(f)|\widetilde{H}=\sthyp\bigl(\Stsym{H}(f)\bigr)|\widetilde{H}=\sthyp(\Stsym{H}(f))\cap\widetilde{H}.
\end{equation*}
From that, it follows
\begin{equation*}\label{e:projhypo}
\proj_H(f)=\proj_H(\Stsym{H}(f))=\Stsym{H}(f)_{|_{H}}
\end{equation*}
where we have used the following notation: for a function $g$ defined in $\R^n$ and $H\in\L^n_{n-1}$, $g_{|_{H}}$ is the restriction of
$g$ to $H$.

On the other hand, by construction of $\Stsym{H}(\cdot)$, it is clear that (for any fixed $h\in H$)
\begin{equation*}
\vol_1\left(x\in h+H^\perp:\,f(x)> t\right)=\vol_1\left(x\in h+H^\perp:\,\Stsym{H}(f)(x)> t\right)
\end{equation*}
and hence
\begin{equation*}
\begin{split}
\int_{\R^n} f\,\dlat x&=\int_H\int_{0}^{\proj_H(f)(h)}\vol_1\left(x\in h+H^\perp:\,f(x)> t\right)\,\dlat t \,\dlat h\\
&=\int_H\int_{0}^{\proj_H(\Stsym{H}(f))(h)}\vol_1\left(x\in h+H^\perp:\,\Stsym{H}(f)(x)> t\right)\,\dlat t \,\dlat h\\
&=\int_{\R^n} \Stsym{H}(f)\,\dlat x.
\end{split}
\end{equation*}

Therefore, we have shown the following result:
\begin{prop}\label{p:propStsym}
Let $f:\R^n\longrightarrow\R_{\geq0}\cup\{\infty\}$ be a non-negative measurable function and let $H\in\L^n_{n-1}$ be a hyperplane.
Then
\begin{enumerate}
\item[(i)]$\Stsym{H}(f):\R^n\longrightarrow\R_{\geq0}\cup\{\infty\}$ is a non-negative measurable function.
\item[(ii)]
$$
\displaystyle\int_{\R^n} f\,\dlat x=\int_{\R^n} \Stsym{H}(f)\,\dlat x.
$$
\item[(iii)]$\proj_H(f)=\proj_H(\Stsym{H}(f))=\Stsym{H}(f)_{|_{H}}$.
\end{enumerate}
\end{prop}

It will be important to relate the Steiner symmetral of the Asplund sum of two non-negative functions with the
Asplund sum of their symmetrals. For this we will need the following inclusion
involving the symmetrals of the (nonempty) measurable sets $A, B$ and $(1-\lambda)A+\lambda B$ respectively. The proof can
be carried out following the ideas of the proof  for convex bodies
(see e.g. \cite[Proposition 9.1]{Gr}); we include it here for completeness.
\begin{prop}\label{p:InclStSymm}
Let $A,B\subset\R^n$ be nonempty measurable sets such that $(1-\lambda)A+\lambda B$ is measurable for a given $\lambda\in(0,1)$ and let $H\in\L^n_{n-1}$ be a hyperplane. Then we have
\begin{equation}\label{e:InclStSymm}
\Stsym{H}\bigl( (1-\lambda)A+\lambda B \bigr)\supset(1-\lambda)\Stsym{H}(A)+\lambda \Stsym{H}(B).
\end{equation}
\end{prop}

\begin{proof}
Let $x\in\Stsym{H}(A)$, $y\in\Stsym{H}(B)$ or, equivalently, $x=h_x+l_x, y=h_y+l_y$ where $h_x,h_y\in H$ and $l_x,l_y\in H^\perp$ are such that
$\enorm{l_x}\leq\frac{1}{2}\vol_1\left(A\cap\bigl(x+H^\perp\bigr)\right)$ and
$\enorm{l_y}\leq\frac{1}{2}\vol_1\left(B\cap\bigl(y+H^\perp\bigr)\right)$. Then $(1-\lambda)x+\lambda y=\bigl((1-\lambda)h_x+\lambda h_y\bigr) + \bigl((1-\lambda)l_x+
\lambda l_y\bigr)$, with $(1-\lambda)h_x+\lambda h_y \in H$ and $(1-\lambda)l_x+\lambda l_y\in H^\perp$. By means of the ($1$-dimensional) Brunn-Minkowski inequality,
we have
\begin{equation*}
\begin{split}
\enorm{(1-\lambda)l_x+\lambda l_y}&\leq(1-\lambda)\enorm{l_x}+\lambda \enorm{l_y}\\
&\leq\frac{1}{2}\left((1-\lambda)\vol_1\bigl(A\cap\bigl(x+H^\perp\bigr)\bigr)
+\lambda\vol_1\bigl(B\cap\bigl(y+H^\perp\bigr)\bigr)\right)\\
&\leq\frac{1}{2}\,\vol_1\left((1-\lambda)\bigl(A\cap\bigl(x+H^\perp\bigr)\bigr)+
\lambda\bigl(B\cap\bigl(y+H^\perp\bigr)\bigr)\right)\\
&\leq \frac{1}{2}\,\vol_1\left(\bigl((1-\lambda)A+\lambda B\bigr)\cap\bigl((1-\lambda)x+\lambda y+H^\perp\bigr)\right).
\end{split}
\end{equation*}
Hence $(1-\lambda)x+\lambda y\in\Stsym{H}\bigl( (1-\lambda)A+\lambda B \bigr)$.
\end{proof}

\subsection{Schwarz-type symmetrization of a function}
The second symmetrization of functions which we will use is defined as follows.
\begin{definition}
Given $H\in\L^n_{n-1}$ and a non-negative measurable function $f:\R^n\longrightarrow\R_{\geq0}\cup\{\infty\}$,
the \emph{symmetrization of} $f$ \emph{with respect to} $H^\perp$ will be the function $\Stsym{H^\perp}(f):\R^n\longrightarrow\R_{\geq0}\cup\{\infty\}$ given by
\begin{equation*}
\Stsym{H^\perp}(f)(h+\alpha \nu)=d_\alpha\chi_{_{B_{n-1}}}(h),
\end{equation*}
for $h\in H$ and $\alpha\in\R$, where
\[d_\alpha=\int_{\alpha \nu+H} f_{|_{\left(\alpha \nu+H\right)}}\,\dlat x,\]
$\nu$ is a normal unit vector of $H$ and $B_{n-1}\in\K^{n-1}$ is the \emph{Euclidean ball of volume} $1$ (lying in $H$).

See Figure \ref{figurasimfunciones2}, where, for a clearer representation, we have made a change of axes in relation to those of Figure \ref{figurasimfunciones}.
\end{definition}

\medskip

\begin{figure}[h]
\begin{adjustwidth}{-20mm}{-20mm}
\hspace{0.7cm}
\includegraphics[width=5.25cm]{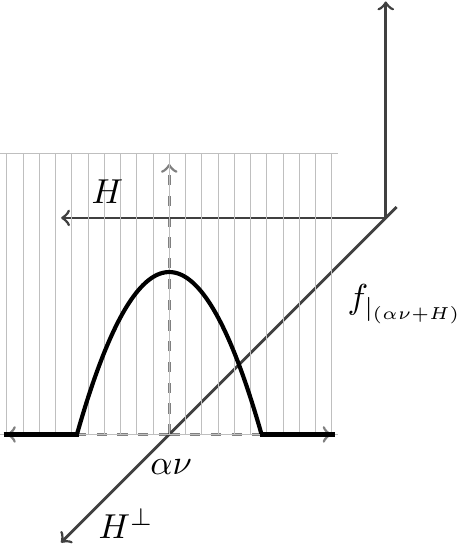}
\hspace{-1.7cm}
\includegraphics[width=5.75cm]{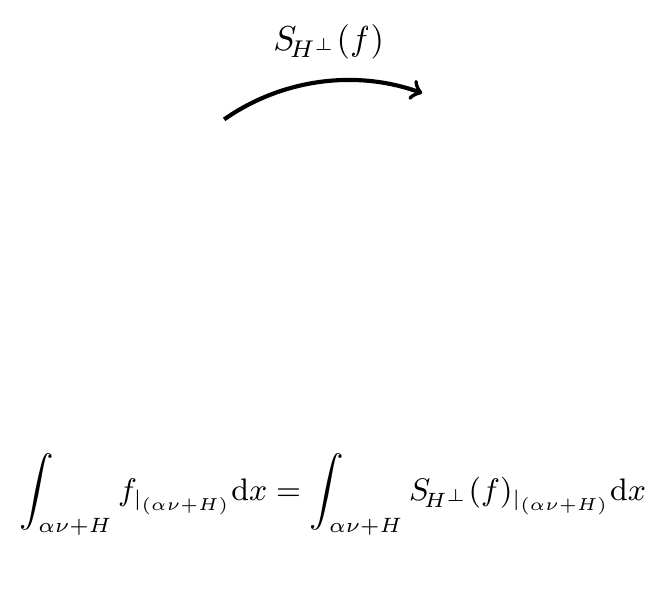}
\hspace{-0.25cm}
\includegraphics[width=6cm]{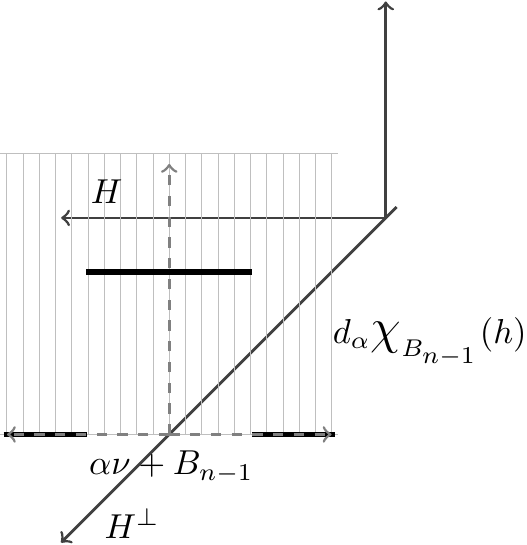}
\caption{}\label{figurasimfunciones2}
\end{adjustwidth}
\end{figure}

\medskip

The choice of the unit volume ball $B_{n-1}$ in the previous definition is not relevant: any other (fixed) convex body (for instance a cube of edge length $1$) could have replaced it
with no essential change. The behavior of the function $\Stsym{H^\perp}(f)$ is basically described by the function $\alpha\longrightarrow d_\alpha$, depending
on the real variable $\alpha$.

By means of Fubini's theorem (together with the fact that the cartesian product of measurable sets is also measurable) it is clear that $\Stsym{H^\perp}(f)$ is also a non-negative measurable function. Moreover we have
\begin{equation}\label{e:Sym2presvol}
\int_{\R^n}\Stsym{H^\perp}(f)\,\dlat x=\int_{\R^n}f\,\dlat x.
\end{equation}
Notice also that the symmetrization $\Stsym{H^\perp}(\cdot)$ is increasing in the sense that if $f\leq g$ then
$\Stsym{H^\perp}(f)\leq\Stsym{H^\perp}(g)$.

In next section (Propositions \ref{p:desigsim_1} and \ref{p:desigsim_2}) we will show that the behavior of these symmetrizations, $\Stsym{H}(\cdot)$, $\Stsym{H^\perp}(\cdot)$, with respect to the operation $\star$ is ``good'' (which may be seen as the analytic counterpart of Proposition \ref{p:InclStSymm}). Roughly speaking we will show that, for both symmetrizations, the symmetral of the Asplund sum is pointwise larger than the Asplund sum of the symmetrals, which will allow us to obtain the inequality of Theorem \ref{t:PLcommon_int_pro}.

\section{Proof of main results}\label{s:main}
In this section we prove Theorem \ref{t:PL-commonproj} and a more general version of Theorem \ref{t:PLcommon_int_pro_log-concave}
(see Theorem \ref{t:PLcommon_int_pro} below). We start by showing a result needed for the proof of Theorem \ref{t:PL-commonproj}.

\begin{prop}\label{p:des_proy_asplund}Let $f,g:\R^n\longrightarrow\R_{\geq0}$ be non-negative functions such that $\proj_H(f),\proj_H(g):H\longrightarrow\R_{\geq0}$ for an $H\in\L^n_{n-1}$ and let $\lambda\in(0,1)$. Then
\begin{equation}\label{e:des_proy_asplund}
\proj_H((1-\lambda) f\star\lambda g)\geq (1-\lambda)\,\proj_H(f)\star\lambda\,\proj_H(g).
\end{equation}
\end{prop}

\begin{proof}
Let $\nu$ be a normal unit vector of $H$. Then, for $h, h_1, h_2\in H$ such that $(1-\lambda) h_1+\lambda h_2=h$ and any $\alpha_1,\alpha_2\in\R$, we clearly have
\begin{equation*}
\begin{split}
\proj_H((1-\lambda) f\star\lambda g)(h)&\geq
(1-\lambda) f\star\lambda g \,\bigl((1-\lambda) (h_1+\alpha_1\nu)+\lambda(h_2+\alpha_2\nu)\bigr)\\
&\geq f(h_1+\alpha_1\nu)^{1-\lambda} g(h_2+\alpha_2\nu)^{\lambda}.
\end{split}
\end{equation*}
Hence, by taking suprema over $\alpha_1,\alpha_2\in\R$, it follows
\begin{equation*}
\proj_H((1-\lambda) f\star\lambda g)(h)\geq \proj_H(f)(h_1)^{1-\lambda} \, \proj_H(g)(h_2)^{\lambda}
\end{equation*}
for all $h_1,h_2\in H$ with $(1-\lambda) h_1+\lambda h_2=h$.
\end{proof}

\begin{proof}[Proof of Theorem \ref{t:PL-commonproj}]
Let $U:H\longrightarrow\R_{\geq0}\cup\{\infty\}$ be the (extended) function given by $U(h)=\proj_H(f)(h)=\proj_H(g)(h)$.
We claim that
\begin{equation}\label{e:aspl-commproj}
\proj_H((1-\lambda) f\star\lambda g)\geq U.
\end{equation}
Indeed, if $U:H\longrightarrow\R_{\geq0}$ then
\begin{equation*}
\proj_H((1-\lambda) f\star\lambda g)\geq (1-\lambda)\,U\star\lambda\,U\geq U  \quad\text{ (cf. \eqref{e:des_proy_asplund})};
\end{equation*}
otherwise following the proof of Proposition \ref{p:des_proy_asplund} (taking $h_1=h_2=h$) we may assert that \eqref{e:aspl-commproj} holds.
Moreover,
\begin{equation}\label{e:f,g_acotados}
\bigl\{x\in h+H^\perp:\,f(x)\geq t\bigr\},\,\bigl\{y\in h+H^\perp:\,g(y)\geq
t\bigr\}\neq\emptyset,
\end{equation}
for all $0\leq t<U(h)$.
By the definition of $\star$ together with \eqref{e:f,g_acotados} it is clear that
\begin{equation*}
\begin{split}
&\bigl\{z\in h+H^\perp\,:\bigl((1-\lambda) f\star\lambda g\bigr)(z)\geq t\bigr\}\\
&\supset(1-\lambda)\bigl\{x\in h+H^\perp\,:f(x)\geq
t\bigr\}+\lambda\bigl\{y\in h+H^\perp\,:g(y)\geq t\bigr\}
\end{split}
\end{equation*}
(for all $0\leq t<U(h)$) and hence, by the Brunn-Minkowski inequality, we obtain
\begin{equation*}
\begin{split}
&\vol_1\left(\bigl\{z\in h+H^\perp\,:\bigl((1-\lambda) f\star\lambda g\bigr)(z)\geq t\bigr\}\right)\\
&\geq(1-\lambda)\vol_1\left(\bigl\{x\in h+H^\perp\,:f(x)\geq t\bigr\}\right) +
\lambda\vol_1\left(\bigl\{y\in h+H^\perp\,:g(y)\geq t\bigr\}\right).
\end{split}
\end{equation*}
From the above inequality, and using Fubini's theorem together with \eqref{e:aspl-commproj}
we get
\begin{equation*}
\begin{split}
&\int_{\R^n}(1-\lambda) f\star\lambda g\,\dlat x \\
&=\int_H\int_{h+H^\perp} \bigl((1-\lambda) f\star\lambda g\bigr)_{|_{\left(h+H^\perp\right)}}\,\dlat x \,\dlat h\\
&=\int_H\int_{0}^{\proj_H((1-\lambda) f\star\lambda g)(h)}\vol_1\left(\left\{x\in h+H^\perp:\,\bigl((1-\lambda) f\star\lambda g\bigr)(x)\geq t\right\}\right)\,\dlat t \,\dlat h\\
&\geq\int_H\int_{0}^{U(h)}\vol_1\left(\left\{x\in h+H^\perp:\,\bigl((1-\lambda) f\star\lambda g\bigr)(x)\geq t\right\}\right)\,\dlat t \,\dlat h\\
&\geq\int_H\int_{0}^{U(h)}(1-\lambda)\,\vol_1\left(\left\{x\in h+H^\perp:\,f(x)\geq t\right\}\right)\,\dlat t \,\dlat h\\
&+\int_H\int_{0}^{U(h)} \lambda\,\vol_1\left(\left\{x\in h+H^\perp:\,g(x)\geq t\right\}\right)\,\dlat t \,\dlat h\\
&=(1-\lambda)\int_{\R^n}f\,\dlat x +\lambda\int_{\R^n}g\,\dlat x,
\end{split}
\end{equation*}
as desired.
\end{proof}

\medskip

Given a non-negative (extended) function $f:\R^n\longrightarrow\R_{\geq0}\cup\{\infty\}$, we will denote by $\overline{f}:\R^n\longrightarrow\R_{\geq0}$ the function given by
\begin{equation*}
\overline{f}(x)=\left\{\begin{array}{ll}
0 &\text{ if  } f(x)=\infty,\\[2mm]
f(x) &\text{ otherwise}.
\end{array}\right.
\end{equation*}
Note that if $f$ is measurable then $\overline{f}$ is measurable as well. Indeed $\overline{f}=f\cdot\chi_F$ where
$F=\{x\in\R^n\,:\,f(x)<\infty\}$; the measurability of  $f$ implies that $F$ is a measurable set, i.e., $\chi_F$
is measurable. Hence $\overline{f}$ is measurable.

\begin{teor}\label{t:PLcommon_int_pro}
Let $f,g:\R^n\longrightarrow\R_{\geq0}$ be non-negative measurable functions
such that $(1-\lambda) f\star\lambda g$ is measurable for $\lambda\in(0,1)$ fixed.
If there exists $H\in\L^n_{n-1}$ such that
\[
(1-\lambda)\overline{\Stsym{H}(f)}\star\lambda\overline{\Stsym{H}(g)}, \quad
(1-\lambda)\Stsym{H^\perp}\bigl(\overline{\Stsym{H}(f)}\bigr)\star\lambda\Stsym{H^\perp}\bigl(\overline{\Stsym{H}(g)}\bigr)
\]
are measurable functions and
\begin{equation}\label{e:hyp_intprojcom}
 \int_H \proj_H(f)(x)\,\dlat x=\int_H \proj_H(g)(x)\,\dlat x<\infty
\end{equation}
then
\begin{equation*}
   \int_{\R^n}(1-\lambda) f\star\lambda g\,\dlat x \geq(1-\lambda)\int_{\R^n}f\,\dlat x+\lambda\int_{\R^n}g\,\dlat x.
\end{equation*}
\end{teor}

We first see how this result implies in particular Theorem \ref{t:PLcommon_int_pro_log-concave}.

\medskip

\begin{proof}[Proof of Theorem \ref{t:PLcommon_int_pro_log-concave}]
If $\phi\,:\,\R^n\longrightarrow\R_{\geq0}$ is a log-concave function satisfying
the condition
\begin{equation*}\label{lemmaCF}
\lim_{|x|\to\infty}\phi(x)=0,
\end{equation*}
then there exist constants $K, a>0$ such that
$$
\phi(x)\le K e^{-a\|x\|}
$$
for every $x\in\R^n$ (see \cite[Lemma 2.5]{CF}). In particular $\phi$ is bounded. It is an easy exercise to check that
$\Stsym{H}(\phi)$ is log-concave as well and, by the boundedness of $\phi$, is also finite, since it is also bounded.
Consequently $\overline{\Stsym{H}(\phi)}=\Stsym{H}(\phi)$. Moreover, by the log-concavity of $\Stsym{H}(\phi)$ and \cite[Corollary 3.5]{BL}, the function
$\alpha\longrightarrow d_\alpha$ in the definition of $\Stsym{H^\perp}\bigl(\overline{\Stsym{H}(\phi)}\bigr)$
is log-concave and then
$\Stsym{H^\perp}\bigl(\overline{\Stsym{H}(\phi)}\bigr)$ is still log-concave, as a product of log-concave functions.

If we apply these considerations to the functions $f$ and $g$ in the statement of the present theorem,
we get that $\overline{\Stsym{H}(f)}$, $\overline{\Stsym{H}(g)}$, $\Stsym{H^\perp}\bigl(\overline{\Stsym{H}(f)}\bigr)$ and $\Stsym{H^\perp}\bigl(\overline{\Stsym{H}(g)}\bigr)$ are log-concave functions. On the other hand the Asplund
sum preserves log-concavity, so that
\[
(1-\lambda)\overline{\Stsym{H}(f)}\star\lambda\overline{\Stsym{H}(g)}, \quad
(1-\lambda)\Stsym{H^\perp}\bigl(\overline{\Stsym{H}(f)}\bigr)\star\lambda\Stsym{H^\perp}\bigl(\overline{\Stsym{H}(g)}\bigr)
\]
are measurable functions. The proof is concluded applying Theorem \ref{t:PLcommon_int_pro}.
\end{proof}

\begin{remark*}
In the statement of the above theorem, we can exchange the condition of decaying to zero at infinity (for both functions $f$ and $g$) for that of boundedness, and we would still obtain the same inequality.
\end{remark*}

In order to establish Theorem \ref{t:PLcommon_int_pro} we need to study the interaction between the symmetrizations $\Stsym{H}(\cdot)$, $\Stsym{H^\perp}(\cdot)$ (for a given
$H\in\L^n_{n-1}$) and the Asplund sum; this is done in Propositions \ref{p:desigsim_1} and \ref{p:desigsim_2}.

\begin{prop}\label{p:desigsim_1}
Let $f,g:\R^n\longrightarrow\R_{\geq0}$ be non-negative measurable functions such that $(1-\lambda) f\star\lambda g$ is measurable for $\lambda\in(0,1)$ fixed and let $H\in\L^n_{n-1}$ be a hyperplane. Then
\begin{equation}\label{e:desigsim_1}
\Stsym{H}((1-\lambda) f\star\lambda g)\geq(1-\lambda)\overline{\Stsym{H}(f)}\star\lambda\overline{\Stsym{H}(g)}.
\end{equation}
\end{prop}

\begin{proof}
Writing $f=e^{-u}$, $g=e^{-v}$, we have that $(1-\lambda) f\star\lambda g=e^{-w}$ with $w=(1-\lambda)u\oplus\lambda v$,
where $\oplus$ denotes the infimal convolution of $u,v$ (see e.g. \cite[p.~34, 38]{Ro}, \cite[Section 1.H]{RoWe}) given by
\begin{equation*}
\bigl((1-\lambda)u\oplus\lambda v\bigr)(t)=\inf\{(1-\lambda)u(x)+\lambda v(y):\, (1-\lambda)x+\lambda y=t\}.
\end{equation*}
The strict epigraph of the infimal convolution satisfies (see e.g. \cite[p.~25]{RoWe})
\begin{equation}\label{e:infconvepi}
\stepi((1-\lambda)u\oplus\lambda v)=(1-\lambda)\stepi(u)+\lambda\stepi(v).
\end{equation}
Write $\overline{\Stsym{H}(f)}=e^{-{\tilde u_H}}$, $\overline{\Stsym{H}(g)}=e^{-{\tilde v_H}}$
(i.e.,  $\tilde u_H=-\log(\overline{\Stsym{H}(f)})$ and $\tilde v_H=-\log(\overline{\Stsym{H}(g)}))$.
We may assume without loss of generality that
both $f$ and $g$ are not identically zero so that $\stepi({\tilde u_H})$, $\stepi({\tilde v_H})\neq\emptyset$. By
\eqref{e:infconvepi}, \eqref{e:InclStSymm} and the definition of $\tilde u_H$, $\tilde v_H$, we have
\begin{equation}\label{e:pruebadesigsim_1}
\begin{split}
\stepi(w_H)&=\Stsym{\widetilde{H}}(\stepi(w))\\
&=\Stsym{\widetilde{H}}((1-\lambda)\stepi(u)+\lambda\stepi(v))\\
&\supset(1-\lambda)\Stsym{\widetilde{H}}(\stepi(u))+\lambda\Stsym{\widetilde{H}}(\stepi(v))\\
&=(1-\lambda)\stepi(u_H)+\lambda\stepi(v_H)\\
&\supset(1-\lambda)\stepi(\tilde u_H)+\lambda\stepi(\tilde v_H)=\stepi((1-\lambda)\tilde u_H\oplus\lambda \tilde v_H).
\end{split}
\end{equation}
This is equivalent to $w_H\leq(1-\lambda)\tilde u_H\oplus\lambda \tilde v_H$; therefore inequality \eqref{e:desigsim_1} holds.
\end{proof}

\begin{remark*} Notice that it was necessary to introduce $\overline{\,\cdot\,}$, as the Asplund sum of two non-negative functions $f$ and $g$ that may attain $\infty$ is in general not well defined. We also point out that the ``$<\!\infty$'' assumption in \eqref{e:hyp_intprojcom} has arisen in order to avoid some ambiguities of the type ``$\infty\cdot0$'' (see the proofs of Proposition \ref{p:desigsim_2} and Theorem \ref{t:PLcommon_int_pro}). All these conflicts will disappear in Section \ref{s:Extension to Borell-Brascamp-Lieb inequalities} when working with the $p$th mean (for $p<0$) instead of the $0$th mean.

On the other hand for real-valued functions,
for instance when working with bounded functions, in the penultimate line of \eqref{e:pruebadesigsim_1} we would have
\begin{equation*}
(1-\lambda)\stepi(u_H)+\lambda\stepi(v_H)=\stepi((1-\lambda)u_H\oplus\lambda v_H),
\end{equation*}
obtaining
$\Stsym{H}((1-\lambda) f\star\lambda g)\geq(1-\lambda)\Stsym{H}(f)\star\lambda\Stsym{H}(g)$.
\end{remark*}

\begin{prop}\label{p:desigsim_2}
Let $f,g:\R^n\longrightarrow\R_{\geq0}$ be non-negative measurable functions such that $(1-\lambda) f\star\lambda g$ is measurable for some fixed
$\lambda\in(0,1)$. Assume also that there exists $H\in\L^n_{n-1}$ such that
\begin{equation*}
\int_{\alpha \nu+H} f_{|_{\left(\alpha \nu+H\right)}}\,\dlat x, \; \int_{\alpha \nu+H} g_{|_{\left(\alpha \nu+H\right)}}\,\dlat x <\infty \;\; \text{for all } \alpha\in\R.
\end{equation*}
Then
\begin{equation*}\label{e:desigsim_2}
\Stsym{H^\perp}((1-\lambda) f\star\lambda g)\geq(1-\lambda)\Stsym{H^\perp}(f)\star\lambda\Stsym{H^\perp}(g).
\end{equation*}
\end{prop}

\begin{proof}
We use the Pr\'ekopa-Leindler inequality (Theorem \ref{t:PrekopaLeindler}).
Indeed, given $\alpha_1,\alpha_2\in\R$ and $\alpha=(1-\lambda)\alpha_1+\lambda\alpha_2$ we clearly have
\begin{equation*}
\begin{split}
&\int_{\alpha \nu+H}\bigl((1-\lambda) f\star\lambda g\bigr)_{|_{\left(\alpha \nu+H\right)}}\,\dlat x\\
&\geq\left(\int_{\alpha_1 \nu+H} f_{|_{\left(\alpha_1 \nu+H\right)}}\,\dlat x\right)^{1-\lambda} \,\left(\int_{\alpha_2 \nu+H} g_{|_{\left(\alpha_2 \nu+H\right)}}\,\dlat x\right)^{\lambda},
\end{split}
\end{equation*}
which allows us to assert that for given $h_1,h_2\in H$
\begin{equation*}
\begin{split}
&\Stsym{H^\perp}\bigl((1-\lambda) f\star\lambda g\bigr)\bigl(((1-\lambda)h_1+\lambda h_2)+((1-\lambda)\alpha_1+\lambda\alpha_2)\nu\bigr)\\
&\geq\Stsym{H^\perp}(f)(h_1+\alpha_1\nu)^{1-\lambda}\,\Stsym{H^\perp}(g)(h_2+\alpha_2\nu)^{\lambda}. \qedhere
\end{split}
\end{equation*}
\end{proof}
As it occurs for the (linear improvements of) Brunn-Minkowski inequality, Theorems \ref{t:BMproye} and \ref{t:BMproye_vol},
the same inequality can be obtained when a condition on the integral of the projection is assumed.

\begin{proof}[Proof of Theorem \ref{t:PLcommon_int_pro}]
Hypothesis \eqref{e:hyp_intprojcom} together with the definition of $\Stsym{H}$, $\Stsym{H^\perp}$ and $\overline{\,\cdot\,}$ implies that (cf. also Proposition \ref{p:propStsym} (iii))
\begin{equation}\label{e:projdoublesymm}
\begin{split}
\proj_H\Bigl(\Stsym{H^\perp}\bigl(\,\overline{\Stsym{H}(f)}\,\bigr)\Bigr)&=
\proj_H\Bigl(\Stsym{H^\perp}\bigl(\Stsym{H}(f)\bigr)\Bigr)\\
=\proj_H\Bigl(\Stsym{H^\perp}\bigl(\Stsym{H}(g)\bigr)\Bigr)
&=\proj_H\Bigl(\Stsym{H^\perp}\bigl(\overline{\Stsym{H}(g)}\bigr)\Bigr).
\end{split}
\end{equation}
Moreover, since the projections onto $H$ of $f$ and $g$ are integrable, by means of Fubini's theorem we have (cf. Proposition \ref{p:propStsym} (ii) and \eqref{e:Sym2presvol})
\begin{equation}\label{e:symspresvol}
\int_{\R^n}\Stsym{H^\perp}\bigl(\,\overline{\Stsym{H}(f)}\,\bigr)\,\dlat x=\int_{\R^n}f\,\dlat x, \quad \int_{\R^n}\Stsym{H^\perp}\bigl(\,\overline{\Stsym{H}(g)}\,\bigr)\,\dlat x=\int_{\R^n}g\,\dlat x.
\end{equation}
On the other hand, Proposition \ref{p:desigsim_1} together with the monotonicity of $\Stsym{H^\perp}(\cdot)$ and
Proposition  \ref{p:desigsim_2} imply
\begin{equation}\label{e:ineqinvolvingsym}
\begin{split}
\Stsym{H^\perp}\bigl(\Stsym{H}((1-\lambda) f\star\lambda g)\bigr)
&\geq\Stsym{H^\perp}\bigl((1-\lambda)\overline{\Stsym{H}(f)}\star\lambda\overline{\Stsym{H}(g)}\,\bigr)\\
&\geq(1-\lambda)\,\Stsym{H^\perp}\bigl(\,\overline{\Stsym{H}(f)}\,\bigr)\star\lambda
\,\Stsym{H^\perp}\bigl(\,\overline{\Stsym{H}(g)}\,\bigr).
\end{split}
\end{equation}
Therefore, applying Proposition \ref{p:propStsym} (ii) together with \eqref{e:Sym2presvol},
\eqref{e:ineqinvolvingsym}, Theorem \ref{t:PL-commonproj} (taking into account \eqref{e:projdoublesymm}) and \eqref{e:symspresvol}, respectively, we get
\begin{equation*}
\begin{split}
  \int_{\R^n}(1-\lambda) f\star\lambda g\,\dlat x
  &=\int_{\R^n}\Stsym{H^\perp}\bigl(\Stsym{H}((1-\lambda) f\star\lambda g)\bigr)\,\dlat x\\
  &\geq\int_{\R^n}(1-\lambda)\,\Stsym{H^\perp}\bigl(\,\overline{\Stsym{H}(f)}\,\bigr)\star\lambda
  \,\Stsym{H^\perp}\bigl(\,\overline{\Stsym{H}(g)}\,\bigr)\,\dlat x\\
  &\geq(1-\lambda)\int_{\R^n}\Stsym{H^\perp}\bigl(\,\overline{\Stsym{H}(f)}\,\bigr)\,\dlat x +\lambda\int_{\R^n}\Stsym{H^\perp}\bigl(\,\overline{\Stsym{H}(g)}\,\bigr)\,\dlat x \\
  &=(1-\lambda)\int_{\R^n}f\,\dlat x+\lambda\int_{\R^n}g\,\dlat x,
\end{split}
\end{equation*}
as desired.
\end{proof}

As a consequence of the above theorem, we may immediately obtain the following refinement of the Brunn-Minkowski inequality (cf. Theorems \ref{t:BMproye} and \ref{t:BMproye_vol}) for the more general case of measurable sets.

\begin{corollary}\label{c:BMproye_vol}
Let $A,B\subset\R^n$ be nonempty measurable sets and let $\lambda\in(0,1)$ be such that $(1-\lambda)A+\lambda B$ is measurable.
If there exists $H\in\L^n_{n-1}$ such that
\begin{enumerate}

  \medskip

  \item[(i)] $\vol_{n-1}(A|H)=\vol_{n-1}(B|H)<\infty,$

  \medskip

  \item[(ii)] $(1-\lambda)\Stsym{H}(A)+\lambda \Stsym{H}(B)$
        is a measurable set,

  \medskip

  \item[(iii)]$(1-\lambda)\Stsym{H^\perp}\bigl(\chi_{_{\Stsym{H}(A)}}\bigr)
             \star\lambda\Stsym{H^\perp}\bigl(\chi_{_{\Stsym{H}(B)}}\bigr)$
        is a measurable function,
  \medskip

\end{enumerate}
then
\begin{equation*}
\vol\bigl((1-\lambda)A+\lambda B\bigr)
\geq(1-\lambda)\vol(A)+\lambda\vol(B).
\end{equation*}
\end{corollary}

\begin{proof}It is enough to consider the functions $f=\chi_{_A}$, $g=\chi_{_B}$ and apply the theorem above (recall that
in this case we have $\proj_H(f)=\chi_{_{A|H}}$, $\proj_H(g)=\chi_{_{B|H}}$ and $(1-\lambda) f\star\lambda g=\chi_{_{(1-\lambda) A+\lambda B}}$
(cf. \eqref{e:AsplCaracFunc}). Notice also that for any measurable set $M$, we have $\Stsym{H}(\chi_{_{M}})=\chi_{_{\Stsym{H}(M)}}$).
\end{proof}

\begin{remark*}


Although the measurability conditions (ii) and (iii) could appear a little bit stronger, they may be also easily fulfilled.
For instance, when working with compact sets $A$ and $B$.
Indeed, since $A$ and $B$ are compact then $\Stsym{H}(A)$ and $\Stsym{H}(B)$ are also compact. Thus condition (ii) holds.
On the other hand, for a general compact set $K$, and by the construction of $\Stsym{H}(\cdot)$ (for sets), the sections of
$\Stsym{H}(K)$ satisfy the following decreasing volume behavior:
\begin{equation*}
\vol_{n-1}\bigl(\Stsym{H}(K)\cap(t_1\nu+H)\bigr)\geq\vol_{n-1}\bigl(\Stsym{H}(K)\cap(t_2\nu+H)\bigr)
\end{equation*}
if $|t_1|\leq |t_2|$ (where $\nu$ is a normal unit vector of $H$). This fact together with the compactness condition of $\Stsym{H}(K)$ imply that
$\Stsym{H^\perp}\bigl(\chi_{_{\Stsym{H}(K)}}\bigr)$ is an upper semi-continuous function.

Now (iii) follows from the fact that
for non-negative upper semi-continuous functions $\phi=e^{-u}$ and $\psi=e^{-v}$ so that $\{x\in\R^n:\, u(x)<\infty\}$, $\{x\in\R^n:\, v(x)<\infty\}$ are bounded, we have (cf. e.g. \cite[p.~25]{RoWe})
\begin{equation*}
\epi((1-\lambda)u\oplus\lambda v)=(1-\lambda)\epi(u)+\lambda\epi(v),
\end{equation*}
where $\epi(\cdot)$ denotes the epigraph of a function.
\end{remark*}

In \cite{BF} (see also \cite[Corollary 1.2.1]{Gi}) a similar result to Theorem
\ref{t:BMproye_vol} was proved, involving sections instead of projections.
The aim of the following result is to prove that the inequality in Theorems \ref{t:PL-commonproj} and \ref{t:PLcommon_int_pro} can be obtained if we replace the projection hypothesis by a suitable section condition.

\begin{teor}\label{t:PL-comvolsect} Let $f,g:\R^n\longrightarrow\R_{\geq0}$ be non-negative measurable functions
such that $(1-\lambda) f\star\lambda g$ is measurable for $\lambda\in(0,1)$ fixed. If there exists
$H\in\L^n_{n-1}$ such that
\begin{equation}\label{e:PL-comvolsect}
\sup_{y\in H^{\bot}}\,\int_{y+H} f_{|_{\left(y+H\right)}}\,\dlat x=
\sup_{y\in H^{\bot}}\,\int_{y+H} g_{|_{\left(y+H\right)}}\,\dlat x<\infty,
\end{equation}
and
$(1-\lambda)\Stsym{H^\perp}(f)\star\lambda\Stsym{H^\perp}(g)$
is a measurable function, then
\begin{equation*}
   \int_{\R^n}(1-\lambda) f\star\lambda g\,\dlat x \geq (1-\lambda)\int_{\R^n}f\,\dlat x +\lambda\int_{\R^n}g\,\dlat x.
\end{equation*}
\end{teor}

\begin{proof}
Notice that hypothesis \eqref{e:PL-comvolsect} together with the definition of $\Stsym{H^\perp}$ implies that
\begin{equation*}
\begin{split}
\int_{H}\proj_H\bigl(\Stsym{H^\perp}(f)\bigr)\,\dlat x&=\sup_{y\in H^{\bot}}\,\int_{y+H} f_{|_{\left(y+H\right)}}\,\dlat x\\
=\sup_{y\in H^{\bot}}\,\int_{y+H} g_{|_{\left(y+H\right)}}\,\dlat x&=\int_{H}\proj_H\bigl(\Stsym{H^\perp}(g)\bigr)\,\dlat x,
\end{split}
\end{equation*}
and, furthermore, we have
\begin{equation*}
\proj_H\bigl(\Stsym{H^\perp}(f)\bigr)=\proj_H\bigl(\Stsym{H^\perp}(g)\bigr).
\end{equation*}
Thus, applying \eqref{e:Sym2presvol},
Proposition \ref{p:desigsim_2} and Theorem \ref{t:PL-commonproj}, respectively, we get
\begin{equation*}
\begin{split}
  \int_{\R^n}(1-\lambda) f\star\lambda g\,\dlat x
  &=\int_{\R^n}\Stsym{H^\perp}\bigl((1-\lambda) f\star\lambda g\bigr)\,\dlat x\\
  &\geq\int_{\R^n}(1-\lambda)\,\Stsym{H^\perp}(f)\star\lambda
  \,\Stsym{H^\perp}(g)\,\dlat x\\
  &\geq(1-\lambda)\int_{\R^n}f\,\dlat x+\lambda\int_{\R^n}g\,\dlat x,
\end{split}
\end{equation*}
as desired.
\end{proof}

\begin{remark*}
We would like to point out that Theorem \ref{t:PL-comvolsect} can be obtained as a particular consequence of results contained in the work \cite{DaUr}, where the authors provide a standard proof of this Borell-Brascamp-Lieb type inequality based on induction in the dimension (cf. \cite[Theorem 3.2]{DaUr}). Indeed, they
obtain a ``more general range" for the parameter $p$ proving that the inequality holds not just for $p\geq-1/n$ but for $p\geq-1/(n-1)$. In Section \ref{s:Extension to Borell-Brascamp-Lieb inequalities} we provide an alternative proof of the inequality in this slightly smaller range based on symmetrization procedures. The work \cite{DaUr} deals with integral inequalities providing simple proofs of certain known inequalities (such as the Pr\'ekopa-Leindler and the Borell-Brascamp-Lieb inequality) as well as new ones. Despite the relevance of the proven inequalities, this work seems not to be so well known in the literature.

We also refer the interested reader to the papers \cite{HMacb,Mars} for related topics involving inequalities for functions and measures.
\end{remark*}

We end this section by remarking that (under a mild assumption on measurability)
the analogous result to Corollary \ref{c:BMproye_vol}
for measurable sets $A$ and $B$ with a common maximal volume section (through parallel hyperplanes to a given one $H$)
can be obtained. This is the content of the following result.

\begin{corollary}
Let $A,B\subset\R^n$ be nonempty measurable sets and $\lambda\in(0,1)$ such that $(1-\lambda)A+\lambda B$ is measurable.
If there exists $H\in\L^n_{n-1}$ such that
\begin{equation*}
\sup_{x\in H^{\perp}}\vol_{n-1}\bigl(A\cap(x+H)\bigr)=\sup_{x\in H^{\perp}}\vol_{n-1}\bigl(B\cap(x+H)\bigr)<\infty,
\end{equation*}
then (provided that
$(1-\lambda)\Stsym{H^\perp}\bigl(\chi_{_{A}}\bigr)\star\lambda\Stsym{H^\perp}\bigl(\chi_{_{B}}\bigr)$
is a measurable function)
\begin{equation*}
\vol\bigl((1-\lambda)A+\lambda B\bigr)
\geq(1-\lambda)\vol(A)+\lambda\vol(B).
\end{equation*}
\end{corollary}

\section{Extension to Borell-Brascamp-Lieb inequalities}\label{s:Extension to Borell-Brascamp-Lieb inequalities}

In this section we present an extension of Theorems \ref{t:PL-commonproj} and
\ref{t:PLcommon_int_pro_log-concave} to the Borell-Brascamp-Lieb (``BBL'' for short)
inequalities. In order to describe these inequalities
and our contribution, we need to introduce some further notation. More precisely we recall the definition of the $p$th mean of two non-negative numbers,
where $p$ is a parameter varying in $\R\cup\{\pm\infty\}$; for this definition we follow \cite{BL} (regarding a general reference for $p$th means of non-negative numbers, we refer also to the classic text of Hardy, Littlewood, and P\'olya \cite{HaLiPo}).

Consider first the case $p\in\R$ and $p\ne0$; given $a,b\ge0$ such that $ab\ne0$ and $\lambda\in[0,1]$, we set
\[
M_p(a,b,\lambda)=((1-\lambda)a^p+\lambda b^p)^{1/p}.
\]
For $p=0$ we set
$$
M_0(a,b,\lambda)=a^{1-\lambda}b^\lambda
$$
and, to complete the picture, for $p=\pm \infty$ we define $M_\infty(a,b,\lambda)=\max\{a,b\}$ and
$M_{-\infty}(a,b,\lambda)=\min\{a,b\}$. Finally, if $ab=0$, we will define $M_p(a,b,\lambda)=0$ for all $p\in\R\cup\{\pm\infty\}$. Note that $M_p(a,b,\lambda)=0$, if $ab=0$, is redundant for all $p\leq0$, however it is relevant for $p>0$ (as we will briefly comment later on). Furthermore, for $p\neq0$, we will allow that $a$, $b$ take the value $\infty$ and in that case, as usual, $M_p(a,b,\lambda)$ will be the value that is obtained ``by continuity''.

The next step is to define a family of functional operations based on these
means, including the Asplund sum for the special case $p=0$. Given non-negative
functions $f,g\,:\,\R^n\longrightarrow\R_{\geq0}$, $p\in\R\cup\{\pm\infty\}$ and $\lambda\in[0,1]$, we define
\begin{equation}\label{e:p-asplund}
(1-\lambda)f\star_p\lambda g \,(x)=
\sup_{(1-\lambda)x_1+\lambda x_2=x}M_p(f(x_1),g(x_2),\lambda).
\end{equation}
In this way, the Asplund sum is obtained for $p=0$. Note further that
\begin{equation}\label{e:inclusion-p-asplund}
(1-\lambda)f\star_p\lambda g\leq(1-\lambda)f\star_q\lambda g
\end{equation}
if $p\leq q$.

The following theorem contains the Borell-Brascamp-Lieb inequality (see
\cite{BL}, \cite{Borell} and also \cite{G} for a detailed presentation).
\begin{theorem}[Borell-Brascamp-Lieb inequality]\label{t:BBL}
Let $\lambda\in(0,1)$, $-1/n\leq p\leq\infty$ and let $f,g:\R^n\longrightarrow\R_{\geq0}$ be non-negative measurable
functions such that $(1-\lambda) f\star_p\lambda g$ is measurable as well. Then
\begin{equation}\label{e:BBL}
\int_{\R^n}(1-\lambda) f\star_p\lambda g\,\dlat x\geq
M_{p/(np+1)}\left(\int_{\R^n}f\,\dlat x,\,\int_{\R^n}g\,\dlat x,\,\lambda\right).
\end{equation}
\end{theorem}

\smallskip

Note that the fact that $M_p(a,b,\lambda)=0$ if $ab=0$ prevents us from obtaining trivial inequalities when $p>0$. Furthermore, the above theorem is a generalization of both the classical Brunn-Minkowski inequality \eqref{BM1} ($p=\infty$ and taking $f$ and $g$ characteristic functions) and Pr\'ekopa-Leindler inequality, Theorem \ref{t:PrekopaLeindler} ($p=0$).

\smallskip

Regarding the functions which are naturally connected to the above theorem, we give the following definition: a non-negative function $f:\R^n\longrightarrow\R_{\geq0}$ is $p$-concave, $p\in\R\cup\{\pm\infty\}$, if
\begin{equation}\label{e:p-concavecondition}
f\bigl((1-\lambda)x_1+\lambda x_2\bigr)\geq M_p\left(f(x_1), f(x_2), \lambda\right)
\end{equation}
for all $x_1,x_2\in\R^n$ and all $\lambda\in(0,1)$.
This definition has the following meaning:
\begin{enumerate}
  \item[(i)] for $p=\infty$, $f$ is $\infty$-concave if and only if $f$ is constant on a convex set and $0$ otherwise;
  \item[(ii)] for $0<p<\infty$, $f$ is $p$-concave if and only if $f^p$ is concave on a convex set and $0$ elsewhere;
  \item[(iii)] for $p=0$, $f$ is $0$-concave if and only if $f$ is log-concave;
  \item[(iv)] for $-\infty<p<0$, $f$ is $p$-concave if and only if $f^p$ is convex;
  \item[(v)] for $p=-\infty$, $f$ is $(-\infty)$-concave if and only if its level sets $\{x\in\R^n:\, f(x)>t\}$ are convex (for all $t\in\R$).
\end{enumerate}
Furthermore, for any $p\in\R\cup\{\pm\infty\}$, $f$ is $p$-concave if and only if
\begin{equation}\label{e:p-concavity-via-p-asplund}
(1-\lambda) f\star_p\lambda f=f.
\end{equation}

\medskip

In the following, for $p\neq0$, we will work with (non-negative) extended measurable functions $f,g:\R^n\longrightarrow\R_{\geq0}\cup\{\infty\}$ for
which we will define the functional operation $\star_p$ as in \eqref{e:p-asplund}. Notice that since $f$ and $g$ can be approximated from below by bounded functions (in such a way that the integrals converge), we are allowed to extend Theorem \ref{t:BBL} for such $f$ and $g$ (functions which may take $\infty$ as a value and provided that $p\neq0$).

In the same way, we will say that a non-negative extended function $f:\R^n\longrightarrow\R_{\geq0}\cup\{\infty\}$ is $p$-concave, $p\neq0$, if and only if the equivalent conditions \eqref{e:p-concavecondition}, \eqref{e:p-concavity-via-p-asplund} hold.

\subsection{BBL inequality under an equal projection assumption}

Let $H\in\L^n_{n-1}$ and let $\nu$ be a normal unit vector of $H$. Given $f,g:\R^n\longrightarrow\R_{\geq0}\cup\{\infty\}$, for $h, h_1, h_2\in H$ such that $(1-\lambda) h_1+\lambda h_2=h$ and any $\alpha_1,\alpha_2\in\R$, we clearly have
\begin{equation*}
\begin{split}
&\,\,\proj_H((1-\lambda) f\star_{-\infty}\lambda g)(h)\\
&\geq\bigl((1-\lambda) f\star_{-\infty}\lambda g\bigr)\bigl((1-\lambda) (h_1+\alpha_1\nu)+\lambda(h_2+\alpha_2\nu)\bigr)\\
&\geq \min \bigl(f(h_1+\alpha_1\nu), \, g(h_2+\alpha_2\nu)\bigr).
\end{split}
\end{equation*}
Thus, by taking suprema over $\alpha_1,\alpha_2\in\R$,
\begin{equation*}
\proj_H((1-\lambda) f\star_{-\infty}\lambda g)(h)\geq \min \bigl(\proj_H(f)(h_1),\, \proj_H(f)(h_2)\bigr)
\end{equation*}
for all $h_1,h_2\in H$ with $(1-\lambda) h_1+\lambda h_2=h$. This implies that
\begin{equation*}
\proj_H((1-\lambda) f\star_{-\infty}\lambda g)\geq (1-\lambda)\,\proj_H(f)\star_{-\infty}\lambda\,\proj_H(g).
\end{equation*}
In particular, if $\proj_H(f)=\proj_H(g)$ and we set $U:H\longrightarrow\R_{\geq0}\cup\{\infty\}$ given by $U(h)=\proj_H(f)(h)=\proj_H(g)(h)$ then
\begin{equation}\label{e:des_proy_asplund_-infty}
\proj_H((1-\lambda) f\star_{-\infty}\lambda g)\geq U.
\end{equation}

On the other hand, it is clear that
\begin{equation}\label{e:f,g_acotadosII}
\bigl\{x\in h+H^\perp:\,f(x)\geq t\bigr\},\,\bigl\{y\in h+H^\perp:\,g(y)\geq
t\bigr\}\neq\emptyset,
\end{equation}
for all $0\leq t<U(h)$ and, by means of the definition of $\star_{-\infty}$ together with \eqref{e:f,g_acotadosII} we have
\begin{equation}\label{e:levelsets-infty}
\begin{split}
&\bigl\{z\in h+H^\perp\,:((1-\lambda) f\star_{-\infty}\lambda g)(z)\geq t\bigr\}\\
&\supset(1-\lambda)\bigl\{x\in h+H^\perp\,:f(x)\geq
t\bigr\}+\lambda\bigl\{y\in h+H^\perp\,:g(y)\geq t\bigr\}
\end{split}
\end{equation}
for all $0\leq t<U(h)$. Now, using the same approach as in the proof of Theorem \ref{t:PL-commonproj}, and taking into account \eqref{e:des_proy_asplund_-infty} and \eqref{e:levelsets-infty}, together with \eqref{e:inclusion-p-asplund}, we may assert:

\begin{teor}\label{t:BBL-commonproj} Let $f,g:\R^n\longrightarrow\R_{\geq0}\cup\{\infty\}$ be non-negative measurable functions
such that $(1-\lambda) f\star_p\lambda g$ is measurable for $p\in\R\cup\{\pm\infty\}$ and $\lambda\in(0,1)$ fixed.
If there exists
$H\in\L^n_{n-1}$ such that
\begin{equation*}
\proj_H(f)=\proj_H(g)
\end{equation*}
then
\begin{equation*}
   \int_{\R^n}(1-\lambda) f\star_p\lambda g\,\dlat x \geq (1-\lambda)\int_{\R^n}f\,\dlat x +\lambda\int_{\R^n}g\,\dlat x.
\end{equation*}
\end{teor}

\smallskip
In the one-dimensional case this theorem was proved by Brascamp and Lieb (see \cite[Theorem~3.1]{BL}).
We notice that Theorem \ref{t:PL-commonproj} is obtained when $p=0$.
Note further that since $p/(np+1)\in[-\infty, 1/n]$, the above inequality is stronger than \eqref{e:BBL} (cf. \eqref{e:inclusion-p-asplund}).

\subsection{BBL inequality under the same integral of a projection}

Given $f:\R^n\longrightarrow\R_{\geq0}\cup\{\infty\}$ measurable and $-\infty<p<0$, it will be convenient to write it in the form $f=u^{1/p}$ where
$u:\R^n\longrightarrow\R_{\geq0}\cup\{\infty\}$ is a measurable function (i.e., $u(x)= f(x)^p$ with the
conventions that $0^p=\infty$ and $\infty^p=0$).
Given an $H\in\L^n_{n-1}$, we may write the Steiner symmetral of $f=u^{1/p}$ in the form
$\Stsym{H}(f)=u_H^{1/p}$ (i.e., $u_H:\R^n\longrightarrow\R_{\geq0}\cup\{\infty\}$ is the function given by $u_H(x)= (\Stsym{H}(f)(x))^p$). Notice further
that, as $t\mapsto t^{1/p}$ is a decreasing bijection on $\R_{\geq0}\cup\{\infty\}$, and $\sthyp(\Stsym{H}(f))=\Stsym{\widetilde{H}}\bigl(\sthyp(f)\bigr)$, we also have
\begin{equation*}
\stepi(u_H)=\Stsym{\widetilde{H}}\bigl(\stepi(u)\bigr).
\end{equation*}

Now, writing $f=u^{1/p}$ and $g=v^{1/p}$, it is easy to check that
\begin{equation*}
(1-\lambda) f\star_p\lambda g=w^{1/p}
\end{equation*}
where $w=(1-\lambda)u\oplus\lambda v$.
Therefore $\Stsym{H}(f)=u_H^{1/p}$, $\Stsym{H}(g)=v_H^{1/p}$ whereas $\Stsym{H}((1-\lambda)f\star_p\lambda g)=w_H^{1/p}$ and, without loss of generality, we may also assume that
both $f$ and $g$ are not identically zero (which implies that $\stepi(u_H)$, $\stepi(v_H)\neq\emptyset$).
Thus, using a similar argument to that at the end of Proposition \ref{p:desigsim_1}, we have
\begin{equation*}
\stepi(w_H)=\Stsym{\widetilde{H}}(\stepi(w))\supset\stepi((1-\lambda) u_H\oplus\lambda  v_H).
\end{equation*}
This is equivalent to $w_H\leq(1-\lambda)u_H\oplus\lambda v_H$ and hence
\begin{equation*}
\Stsym{H}((1-\lambda) f\star_p\lambda g)=w_H^{1/p}\geq \bigl((1-\lambda)u_H\oplus\lambda v_H\bigr)^{1/p}
=(1-\lambda) \Stsym{H}(f)\star_p\lambda \Stsym{H}(g).
\end{equation*}

Furthermore, notice that for $p=-\infty$, we have
\begin{equation*}
\begin{split}
&\bigl\{z\in \R^n\,: \Stsym{H}((1-\lambda) f\star_{-\infty}\lambda g)(z)> t\bigr\}\\
&=\Stsym{H}\left(\bigl\{z\in\R^n\,: ((1-\lambda) f\star_{-\infty}\lambda g)(z)> t\bigr\}\right)\\
&\supset(1-\lambda)\Stsym{H}\left(\bigl\{x\in\R^n\,:f(x)>t\bigr\}\right)
+\lambda\Stsym{H}\left(\bigl\{y\in\R^n\,:g(y)>t\bigr\}\right)\\
&=(1-\lambda)\bigl\{x\in\R^n\,:\Stsym{H}(f)(x)>t\bigr\}
+\lambda\bigl\{y\in\R^n\,:\Stsym{H}(g)(y)> t\bigr\}\\
&=\bigl\{z\in \R^n\,: ((1-\lambda)  \Stsym{H}(f)\star_{-\infty}\lambda  \Stsym{H}(g))(z)> t\bigr\},
\end{split}
\end{equation*}
and thus $\Stsym{H}((1-\lambda) f\star_{-\infty}\lambda g)\geq(1-\lambda) \Stsym{H}(f)\star_{-\infty}\lambda \Stsym{H}(g)$.

\medskip

On the other hand,
given $\alpha_1,\alpha_2\in\R$ and $\alpha=(1-\lambda)\alpha_1+\lambda\alpha_2$,
and by means of the Borell-Brascamp-Lieb inequality (Theorem \ref{t:BBL}), we  have
\begin{equation*}
\begin{split}
&\int_{\alpha \nu+H}\bigl((1-\lambda) f\star_p\lambda g\bigr)_{|_{\left(\alpha \nu+H\right)}}\,\dlat x\\
&\geq M_{p/(np+1)}\left(\int_{\alpha_1 \nu+H} f_{|_{\left(\alpha_1 \nu+H\right)}}\,\dlat x,\,\int_{\alpha_2 \nu+H} g_{|_{\left(\alpha_2 \nu+H\right)}}\,\dlat x,\, \lambda \right),
\end{split}
\end{equation*}
for $-1/n\leq p<0$. This ensures that, for $h_1,h_2\in H$,
\begin{equation*}
\begin{split}
&\Stsym{H^\perp}\bigl((1-\lambda) f\star_p\lambda g\bigr)\bigl(((1-\lambda)h_1+\lambda h_2)+((1-\lambda)\alpha_1+\lambda\alpha_2)\nu\bigr)\\
&\geq M_{p/(np+1)}\bigl(\Stsym{H^\perp}(f)(h_1+\alpha_1\nu),\,\Stsym{H^\perp}(g)(h_2+\alpha_2\nu),\,\lambda\bigr).
\end{split}
\end{equation*}

\smallskip

Therefore, we have shown the following result:
\begin{prop}\label{p:BBL-inclu-sym}
Let $f,g:\R^n\longrightarrow\R_{\geq0}\cup\{\infty\}$ be non-negative measurable functions such that $(1-\lambda) f\star_p\lambda g$ is measurable for $\lambda\in(0,1)$ fixed and $-\infty\leq p<0$. Then, given
$H\in\L^n_{n-1}$,
\begin{equation}\label{e:incSteiner-pAsplund}
\Stsym{H}((1-\lambda) f\star_p\lambda g)\geq(1-\lambda)\Stsym{H}(f)\star_p\lambda\Stsym{H}(g).
\end{equation}
Moreover, if $-1/n\leq p<0$,
\begin{equation}\label{e:incSymII-pAsplund}
\Stsym{H^\perp}((1-\lambda) f\star_p\lambda g)\geq(1-\lambda)\Stsym{H^\perp}(f)\star_{q}\lambda\Stsym{H^\perp}(g),
\end{equation}
where $q=p/(np+1)$.
\end{prop}

Now, as in Theorem \ref{t:PLcommon_int_pro}, we extend the above theorem for the case of two functions with the same integral of a projection onto a hyperplane.

\begin{teor}\label{t:BBLcommon_int_pro}
Let $-1/n\leq p<0$ and let $f,g:\R^n\longrightarrow\R_{\geq0}\cup\{\infty\}$ be non-negative measurable functions such that
$(1-\lambda) f\star_p\lambda g$ is measurable for $\lambda\in(0,1)$ fixed.
If there exists $H\in\L^n_{n-1}$ such that
\[
(1-\lambda)\Stsym{H}(f)\star_p\lambda\Stsym{H}(g), \quad
(1-\lambda)\Stsym{H^\perp}\bigl(\Stsym{H}(f)\bigr)\star_{-\infty}\lambda\Stsym{H^\perp}\bigl(\Stsym{H}(g)\bigr)
\]
are measurable functions and
\begin{equation}\label{e:BBL-hyp_intprojcom}
 \int_H \proj_H(f)(x)\,\dlat x=\int_H \proj_H(g)(x)\,\dlat x
\end{equation}
then
\begin{equation*}
   \int_{\R^n}(1-\lambda) f\star_p\lambda g\,\dlat x
   \geq(1-\lambda)\int_{\R^n}f\,\dlat x +\lambda\int_{\R^n}g\,\dlat x.
\end{equation*}
\end{teor}

\begin{proof}
Hypothesis \eqref{e:BBL-hyp_intprojcom} together with the definition of $\Stsym{H}$, $\Stsym{H^\perp}$ implies
\begin{equation*}
\proj_H\Bigl(\Stsym{H^\perp}\bigl(\Stsym{H}(f)\bigr)\Bigr)
=\proj_H\Bigl(\Stsym{H^\perp}\bigl(\Stsym{H}(g)\bigr)\Bigr).
\end{equation*}

On the other hand, Proposition \ref{p:BBL-inclu-sym} (together with the monotonicity of $\Stsym{H^\perp}(\cdot)$ and \eqref{e:inclusion-p-asplund}) implies
\begin{equation*}\label{e:BBLineqinvolvingsym}
\begin{split}
\Stsym{H^\perp}\bigl(\Stsym{H}((1-\lambda) f\star_p\lambda g)\bigr)
&\geq\Stsym{H^\perp}\bigl((1-\lambda)\Stsym{H}(f)\star_p\lambda\,\Stsym{H}(g)\bigr)\\
&\geq(1-\lambda)\,\Stsym{H^\perp}\bigl(\Stsym{H}(f)\bigr)\star_q\lambda
\,\Stsym{H^\perp}\bigl(\Stsym{H}(g)\bigr)\\
&\geq(1-\lambda)\,\Stsym{H^\perp}\bigl(\Stsym{H}(f)\bigr)\star_{-\infty}\lambda
\,\Stsym{H^\perp}\bigl(\Stsym{H}(g)\bigr),\\
\end{split}
\end{equation*}
where $q=p/(np+1)$.
The proof is now concluded by following similar steps to Theorem \ref{t:PLcommon_int_pro}.
\end{proof}

\begin{remark*}
Notice that if in the above theorem $f,g:\R^n\longrightarrow\R_{\geq0}\cup\{\infty\}$ are such that $(1-\lambda) f\star_{p^\prime}\lambda g$ is measurable for $\lambda\in(0,1)$ and $p^\prime\geq0$, then (cf. \eqref{e:inclusion-p-asplund})
\begin{equation*}
   \int_{\R^n}(1-\lambda) f\star_{p^\prime}\lambda g\,\dlat x
   \geq(1-\lambda)\int_{\R^n}f\,\dlat x +\lambda\int_{\R^n}g\,\dlat x.
\end{equation*}
\end{remark*}

\smallskip

\begin{proof}[Proof of Theorem \ref{t:BBLcommon_int_pro_p-concave}]
Without loss of generality we may assume that $p<0$ (cf. \eqref{e:inclusion-p-asplund}).
It is an easy exercise to check that if a function $\phi$ is $p$-concave then $\Stsym{H}(\phi)$ and $\Stsym{H^\perp}(\phi)$ are, respectively, $p$-concave and $q$-concave ($q=p/(np+1)$) functions
(in fact, this may be quickly obtained from \eqref{e:p-concavity-via-p-asplund}, \eqref{e:incSteiner-pAsplund} and \eqref{e:incSymII-pAsplund}).

On the other hand, since $q$-concave functions are also $(-\infty)$-concave and the $t$-Asplund
sum, $\star_t$, preserves $t$-concavity, we may assert that
\[
(1-\lambda)\Stsym{H}(f)\star_p\lambda\Stsym{H}(g), \quad
(1-\lambda)\Stsym{H^\perp}\bigl(\Stsym{H}(f)\bigr)\star_{-\infty}\lambda\Stsym{H^\perp}\bigl(\Stsym{H}(g)\bigr)
\]
are measurable functions. The proof is concluded by applying Theorem \ref{t:BBLcommon_int_pro}.
\end{proof}

To end this paper, we establish here the analogous result to Theorem \ref{t:PL-comvolsect}; it can be shown following the steps of the proof of the above-mentioned theorem.

\begin{teor}\label{t:BBL-comvolsect} Let $-1/n\leq p<0$ and let $f,g:\R^n\longrightarrow\R_{\geq0}\cup\{\infty\}$ be non-negative measurable functions such that $(1-\lambda) f\star_p\lambda g$ is measurable for $\lambda\in(0,1)$ fixed. If there exists $H\in\L^n_{n-1}$ such that
\begin{equation*}\label{e:BBL-comvolsect}
\sup_{y\in H^{\bot}}\,\int_{y+H} f_{|_{\left(y+H\right)}}\,\dlat x=
\sup_{y\in H^{\bot}}\,\int_{y+H} g_{|_{\left(y+H\right)}}\,\dlat x,
\end{equation*}
and
$(1-\lambda)\Stsym{H^\perp}(f)\star_{-\infty}\lambda\Stsym{H^\perp}(g)$
is a measurable function, then
\begin{equation*}
   \int_{\R^n}(1-\lambda) f\star_p\lambda g\,\dlat x \geq (1-\lambda)\int_{\R^n}f\,\dlat x +\lambda\int_{\R^n}g\,\dlat x.
\end{equation*}
\end{teor}

\bigskip

\noindent {\it Acknowledgements.}
This work was developed during a research stay of
the third author at the Dipartimento di Matematica ``U. Dini'', Firenze, Italy, supported by ``Ayudas para estancias breves de la Universidad de Murcia'', Spain.

The authors thank the anonymous referee for the careful reading of the paper and very useful suggestions which significantly improved the presentation.

\end{document}